\documentclass[12pt]{elsarticle}
\usepackage{amsfonts}
\usepackage{amsmath,cancel}
\usepackage{amssymb}
\usepackage{url}
\usepackage[usenames]{color}
\usepackage{graphicx}
\usepackage{mathrsfs}
\usepackage[normalem]{ulem}
\newtheorem{problem}{Problem}
\newtheorem{theorem}{Theorem}[section]

\newtheorem{teo}[theorem]{Theorem}

\newtheorem{lema}[theorem]{Lemma}

\newtheorem{cor}[theorem]{Corollary}

\newtheorem{definition}[theorem]{Definition}
\newtheorem{example}[theorem]{Example}
\newtheorem{lemma}[theorem]{Lemma}

\newtheorem{prop}[theorem]{Proposition}

\newenvironment{proof}[1][\noindent Proof]{\textbf{#1.} }{\ \rule{0.5em}{0.5em}}

\newcommand{\R}{\mathbb{R}}
\newcommand{\M}{\mathcal{M}}

\title{Complementarity spectrum of digraphs}

\author{Diego Bravo}
\ead{dbravo@fing.edu.uy }
\author{Florencia Cubr\'{\i}a\footnote{Corresponding author: fcubria@fing.edu.uy. \\ Published in \textit{Linear Algebra and its Applications}, Vol. 627, 2021, Pages 24-40.\\ \url{https://doi.org/10.1016/j.laa.2021.06.001}}}
\ead{fcubria@fing.edu.uy}
\author{Marcelo Fiori}
\ead{mfiori@fing.edu.uy}
\address{Instituto de Matemática y Estadística ``Rafael Laguardia'', Facultad de Ingenier\'{i}a,\\ Universidad de la Republica, Uruguay.}

\author{Vilmar Trevisan}
\ead{mfiori@fing.edu.uy}
\address{UFRGS – Universidade Federal do Rio Grande do Sul,
Instituto de Matemática, Porto Alegre, Brazil, and
Dipartimento di Matematica e Applicazioni, University of Naples Federico
II, Italy}

\date{June 7, 2021}
\begin{document}

\begin{abstract}
In this paper we study the complementarity spectrum of digraphs, with special attention to the problem of digraph characterization through this complementarity spectrum. That is, whether two non-isomorphic digraphs with the same number of vertices can have the same complementarity eigenvalues. The complementarity eigenvalues of matrices, also called Pareto eigenvalues, has led to the study of the complementarity spectrum of (undirected) graphs and, in particular, the characterization of undirected graphs through these eigenvalues is an open problem. We characterize the digraphs with one and two complementarity eigenvalues, and we give examples of non-isomorphic digraphs with the same complementarity spectrum.
\end{abstract}
\maketitle
\section{Introduction}

The concept of complementarity eigenvalues for matrices was introduced by Seeger in \cite{Seeger99} and has found many applications in different fields of science. The complementarity spectrum of a graph $G$ has been defined as the complementarity spectrum of the adjacency matrix of $G$ by Fernandes et. al in \cite{Fernandes2017}, while Seeger \cite{Seeger2018} has suggested representing an undirected graph by its complementarity spectrum. The main purpose of this paper is to introduce the concept of complementarity spectrum of digraphs. This will be naturally defined as the complementarity spectrum of the adjacency matrix of the digraph.

The main outcome of the present manuscript is the interpretation of the complementarity eigenvalues of a digraph in terms of its structural properties. In order to explain the results we obtain, we briefly review some known facts about (di)graph determination.

{It is well known that the Graph Isomorphism Problem is computationally hard and hence, because computing the spectra of (di)graphs can be done in polynomial time, using spectral parameters as tools to characterize (di)graphs, is a field of study that is increasing in mathematical sciences. Already in 1957, Collatz and Sinogowitz \cite{Collatz1957} have shown that there are  non isomorphic cospectral graphs. This means that graphs are not determined by their spectra - not DS for short. In fact that are several classes of examples that may be found, for example, in \cite{Cvetkovic1998}, which include trees, graphs of arbitrary size, and digraphs \cite{Harary1971}. An active research problem is to determine whether most graph are DS \cite{Haemers}, meaning that the proportion of graphs having a cospectral mate among all graphs of the same order $n$ (number of vertices) is arbitrarily small when $n$ goes to infinity.}

{Adapting the definition given in \cite{Pinheiro2020} for graphs, we say that a strongly connected digraph is \textit{determined by its complementarity spectrum} -- DCS for short --  if any complementarity cospectral digraph is either isomorphic, or has a different order.}

{It is worth pointing out that no examples of connected non isomorphic graphs, with the same order, having the same complementarity spectrum have being found so far \cite{Pinheiro2020}. This is a strong evidence that the complementarity spectrum distinguishes graphs in a better way than other any usual spectra. A hard problem is to determine whether all graphs are DCS. In a sharp contrast, we find here (see Section \ref{sec:contraejemplos}) families of non isomorphic digraphs with equal order that are complementarity cospectral.}

{Thus one of the main results of this note is that not all digraphs are DCS. With respect to the problem of digraph characterization, this result implies that the complementarity spectrum is not a better tool than, say, any other spectrum. On the other hand, the complementarity spectrum of digraphs determines some interesting structural properties.}

{As an example, we show that the complementarity eigenvalues are spectral radii of strongly connected subdigraphs. Additionally, we characterize the digraphs having one or two complementarity eigenvalues, showing that there is a strong connection between the cyclic structure of the digraph and the cardinality of its complementarity spectrum (see Section \ref{sec:characterization} for details).}

{We emphasize that the goal of this note is to present the concept of complementarity spectrum of digraphs and a long term purpose of our research is understand in general how this spectra may reveal structural properties of digraphs. Since this is too broad for a unique paper, we focus here on the problem of digraph characterization and develop some fundamental theory about the complementarity spectra of digraphs. In particular, we anticipate that the cardinality of the complementarity spectra is key for understanding what the complementarity spectra of digraphs may reveal of its structure. We refer to our Section \ref{sec:final} for some open problems.}

The remaining  of the paper is organized as follows.
In Section \ref{sec:prelim} we recall some definitions, known results, and we set the notation used throughout the paper.
In Section \ref{sec:matrices} we recall the definition of complementarity eigenvalues for matrices, and we prove some useful results for the setting of this paper.
In Section \ref{sec:digraphs} we introduce the concept of complementarity spectrum for digraphs, and present some results regarding it. The main result in this section is the characterization of the complementarity spectrum of a digraph in terms of its structure. In Section \ref{sec:characterization} we completely characterize the digraphs with at most two complementarity eigenvalues.
Finally, in Section \ref{sec:contraejemplos} we study some specific families of digraphs with three complementarity eigenvalues. The main result in this section is the presentation of examples of non isomorphic digraphs of arbitrary order that are cospectral.

\section{Preliminaries}

The vector \textbf{o} will denote the null vector in $\R^n$, the vector \textbf{1} will denote the vector in $\R^n$ such that $\textbf{1}_i=1$ for all $i=1, \hdots, n$, and when comparing two vectors, $\leq$ and $<$ means that the inequality holds for every coordinate. Given a set $S$, its cardinality is denoted as $\#S$.

\label{sec:prelim}
Let $D=(V,E)$ be a finite simple digraph and vertices labelled as $1, \hdots, n$. The adjacency matrix of $D$ is defined as $A(D)=(a_{ij})$ where
\[a_{ij}=
\begin{cases}
1 \quad \text{if }(i,j)\in E,\\
0 \quad \text{otherwise.}
\end{cases}
\]

Let $p_D$ denote the characteristic polynomial of $A(D)$. The multiset of roots of $p_D$, counted with their multiplicities, is the spectrum of $D$.

We refer to the landmark paper by R. Brualdi \cite{BRUALDI2010} for a comprehensive analysis of spectra of digraphs.

{Throughout this paper, $\rho(\cdot)$ denotes the spectral radius (i.e., the largest module of the eigenvalues) of a matrix. For nonnegative irreducible matrices, it is well known that the spectral radius coincides with the largest eigenvalue, due to the powerful Perron-Frobenius Theorem. Additionally, the spectral radius is simple and may be associated with an eigenvector $x>\textbf{o}$. On the other hand, it is well known that any adjacency matrix $A(D)$ of a digraph $D$ is irreducible if and only if $D$ is strongly connected. Hence, we have that
$$\rho(A(D))= \max \{\rho(A(D_i))~ |~ D_i \mbox{~is a strongly connected component of ~} D  \}.$$
This real positive value $\rho(A(D))$ is called the \emph{spectral radius of the digraph $D$} and is denoted by $\rho(D)$. As we will see, the spectral radius of the digraph plays a fundamental role on the results we obtain in this note.}

A digraph $H=(V',E')$ is a subdigraph of $D$ (denoted $H\leq D$) if $V'\subset V$ and $E'\subset E$. We say that $H$ is an induced subdigraph if $E'=E \cap (V'\times V')$ and a proper subdigraph if $E'\neq E$. 

{The following is well known, but we state here for easy reference.
\begin{lemma}\label{lem:sub} Let $H$ be a proper subdigraph of a strongly connected digraph $D$. Then $\rho(H) < \rho(D)$.
\end{lemma}
}

We use the term \textit{cycle} to refer to a directed cycle in a digraph. A disjoint union of cycles will be called a \emph{linear digraph}.

The following result is central in the spectral theory of digraphs,  allowing one to compute the characteristic polynomial in terms of the structural properties of the digraph:

\begin{teo}[Sachs \cite{Sachs1964}]\label{tm:sach}
Let $P_D(x)=x^n+a_1x^{n-1}+\hdots +a_{n-1}x+a_n$ be the characteristic polynomial of $D$, then,
\[a_i=\sum_{L\in \mathscr{L}_i } (-1)^{p(L)},\]
where
\[\mathscr{L}_i=\{ \text{linear subdigraphs of $D$ with $i$  vertices}\},\]
\[p(L)= \# \{ \text{strongly connected components of } L\}.\]
\end{teo}

As an example, we apply this theorem in order to compute the characteristic polynomial and spectral radius of the cycle digraph $\vec{C}_n$, which will be useful in what follows.

\begin{example} The digraph $\vec{C}_n$ has a unique linear subdigraph: $\vec{C}_n$ itself. Then $a_i=0$ for all $i\neq n$ and $a_n=\displaystyle\sum_{L\in \mathscr{L}_n } (-1)^{p(L)}=(-1)^{p(\vec{C}_n)}=-1$. Hence, we have that
\[p_{\vec{C}_n}(x)=x^n-1,\]
whose zeros are the $n$-th roots of unity. So, $\rho(\vec{C}_n)=1$.
\end{example}

\section{Complementarity eigenvalues of matrices}
\label{sec:matrices}
The Eigenvalue Complementarity Problem (EiCP) introduced in \cite{Seeger99} has found many applications in different fields of science, engineering and economics \cite{Adly2015,Facchinei2007,Pinto2008,Pinto2004}.

Given a matrix $A \in \M_n(\mathbb{R})$, the set of complementarity eigenvalues is defined as those $\lambda \in \R$ such that there exist a vector $x \in \R^n$, not null and nonnegative, verifying
 $Ax \geq \lambda x$, and
\begin{equation*}
\label{complement}
\langle  x, Ax-\lambda x \rangle =0.
\end{equation*}

If we write $w=Ax-\lambda x\geq \textbf{o}$, the previous condition results in
\[x^tw=0 \]
which means to ask for
\[x_i=0 \quad \text{or} \quad w_i=0 \text{ for all } i=1 \hdots, n.\]

This last condition is called \textit{complementarity condition}.

The set of all complementarity eigenvalues of a matrix $A$ is called the \textit{complementarity spectrum} of $A$, and it is denoted $\Pi(A)$. Unlike the regular spectrum of a matrix, the complementarity spectrum is a set, and the number of complementarity eigenvalues is not determined by the size of the matrix.

It is known that if $\lambda$ is a complementarity eigenvalue of $A$, then it is a complementarity eigenvalue of $PAP^t$ as well, for every permutation matrix $P$ \cite{Pinto2008}.

This fact allows us to define the complementarity spectrum of a digraph, since the complementarity spectrum is invariant in the family of adjacency matrices associated to the digraph.

In what follows, we recall a series of results that will allow us to characterize the complementarity eigenvalues of a digraph in terms of its structural properties.

\begin{teo}\cite{Seeger99}
\label{submatriz}
Let $A \in \M_n(\R)$. If $\lambda$ is a complementarity eigenvalue of $A$ then $\lambda$ is an eigenvalue of some principal submatrix of $A$ with an associated eigenvector $x>\textbf{o}$.
The converse statement is true if the off-diagonal entries of A are non-negative.
\end{teo}

The following result is a reinterpretation of Theorem \ref{submatriz} in terms of the spectral radius of principal submatrices.

\begin{prop}
\label{prop:submatriz}
Let $A\in \mathcal{M}_n(\R)$ be a non-negative matrix, then
\begin{equation}
\label{obssubmatriz}
\Pi(A) =\{\rho(B):  B \text{ is a principal submatrix of $A$ irreducible or null}\}.
\end{equation}
\end{prop}
\begin{proof}
If $\lambda$ is a complementarity eigenvalue of $A$, by Theorem \ref{submatriz} we have that $\lambda$ is as well an eigenvalue of $B$ (a principal submatrix of $A$) with an eigenvector $x>\textbf{o}$ associated.

Let us first consider the case in which $B$ is irreducible. From the fact that $x>\textbf{o}$ is an eigenvector of $B$ associated with $\lambda$, by virtue of the Perron-Frobenius Theorem we know that  $\lambda=\rho(B)$ with $B$ a principal irreducible submatrix of $A$.

Consider now the case where $B$ is reducible, without loss of generality due to our previous observation regarding conjugation by permutation matrices, we can assume that
\[B=
\left(
\begin{array}{cc}
X & Y   \\
\textbf{0} & Z
\end{array}
\right),\]
with $X \in \mathcal{M}_r(\R)$, $Y \in \mathcal{M}_{r \times s}(\R)$, and $Z \in \mathcal{M}_s(\R)$, all nonnegative, and  $Z$ irreducible or null. It is easy to see that $\lambda$ is an eigenvalue of $Z$ with an eigenvector $x'>\textbf{o}$ associated. If $Z$ is irreducible, as the previous case we have that $\lambda=\rho(Z)$ with $Z$ principal irreducible submatrix of $A$; while if $Z$ is null we have that $\lambda=0=\rho(Z)$ with $Z$ principal null submatrix if $A$.

Let $\rho(B)$ be the spectral radius of $B$ principal submatrix of $A$ irreducible or null. In case $B$ is irreducible, by Perron Frobenius Theorem there exists an eigenvector $x>\textbf{o}$ associated with $\rho(B)$ while if $B$ is null we know that $\textbf{1}>\textbf{o}$ is an eigenvector associated with $0=\rho(B)$. In both cases, due to Theorem \ref{submatriz}, we have that $\rho(B)$ belongs to $\Pi(A)$.
\end{proof}


\begin{cor}
\label{cor:Dt}
{Let $A\in \mathcal{M}_n(\R)$ be a non-negative matrix, then
$\Pi(A) =\Pi(A^T).$}
\end{cor}
\begin{proof}
{\[\Pi(A) =\{\rho(B):  B \text{ is a principal submatrix of $A$ irreducible or null}\}=\]
\[\{\rho(A_{JJ}):  J \subset \{1, \hdots, n\}, A_{JJ} \text{ irreducible or null }\}=\]
\[\{\rho((A_{JJ})^T):  J \subset \{1, \hdots, n\},  (A_{JJ})^T\text{ irreducible or null }\}=\]
\[\{\rho((A^T)_{JJ}):  J \subset \{1, \hdots, n\},   (A^T)_{JJ} \text{ irreducible or null }\}=\Pi(A^T).
\]}
\end{proof}

In \cite{Pinto2008}, the authors observe that, in general, the complementarity spectrum of a matrix may not coincide with the complementarity spectrum of its transpose. In this case, Corollary \ref{cor:Dt} shows that they do coincide for non-negative matrices. This observation will be usefull in the context of digraphs in this paper.

In the following section we will use these results for complementarity eigenvalues of matrices, in order to define the complementarity spectrum of a digraph, and study some properties regarding this spectrum.

\section{Complementarity eigenvalues of digraphs}
\label{sec:digraphs}

Let us first define the complementarity spectrum of a finite simple digraph, i.e., a digraph with no multiple arcs nor self-loops.

\begin{definition}
Let $D=(V,E)$ be a simple digraph and $A=A(D)$ its adjacency matrix. The complementarity spectrum of $D$, denoted as $\Pi(D)$, is the complementarity spectrum of its adjacency matrix $A$.
\end{definition}

Two digraphs are said to be \textit{complementarity cospectral} if they have the same complementarity spectrum.

Note that this spectrum is well defined since, as we noted above, the complementarity spectrum of a matrix is invariant in the family of adjacency matrices of $D$.

For the sake of simplicity, given a digraph $D$ with $n$ vertices, we will suppose that the set of vertices is $V=\{1, \hdots, n\}$. 

The following results lead to a simple and useful characterization of the complementarity eigenvalues for digraphs.

\begin{teo}
\label{compl_spect_induced_subdigraphs} Let $D$ be a digraph and $\Pi(D)$ its complementarity spectrum. Then
\[\Pi(D)=\{\rho(H): {H \text{ induced strongly connected subdigraph of }D}\}.\]
\end{teo}
\begin{proof}
This follows from Proposition \ref{prop:submatriz}, and from the fact that $B=A_{JJ}=A(H)$, where $H$ is the digraph induced by the vertices in $J$. If $B$ is irreducible then $H$ is strongly connected. If $B$ is null then $H$ is composed of isolated vertices, in which case we have that $\rho(H)=\rho(H')$ where $H'$ the subdigraph generated by a single vertex which is strongly connected.
\end{proof}

By the previous characterization, we can see that the complementarity spectrum of a digraph is a nonnegative set, which always contain zero as an element.

This result extends the characterization of the complementarity spectrum for graphs, given in \cite{Fernandes2017}.
The complementarity spectrum can also be expressed in terms of the complementarity spectrum of its strongly connected components. Indeed, we have the following result.

\begin{prop}\label{prop:connect_components}
Let $D$ be a digraph and $D_1, \hdots, D_k$ the digraph generated by the strongly connected components of $D$. Then,
\[
\Pi(D)=\cup_{i=1}^k \Pi(D_i),
\]
\end{prop}
\begin{proof}
By Theorem \ref{compl_spect_induced_subdigraphs}, we have to prove that:
\[\{\rho(H): {H \text{ induced strongly connected subdigraph of }D}\}=\]
\[\cup_{i=1}^k \{\rho(H): {H \text{ induced strongly connected subdigraph of }D_i}\}\]

This is a consequence of the fact that $H$ is a strongly connected subdigraph of $D$ if and only if $H$ is a strongly connected subdigraph of $D_i$ for some $i$.
\end{proof}

{From Proposition~\ref{prop:connect_components}, it follows that, in general, there is no loss in generality in assuming that $D$ is strongly connected. Furthermore it is easy to construct pairs of digraphs, not strongly connected, with equal number of vertices and arcs having the same complementarity spectrum. Indeed, consider two different (non isomorphic) digraphs $D_1=(V_1, E_1)$ and $D_2=(V_2, E_2)$, and define $D$ and $H$ as follows.
\[D= D_1 \cup D_2 \cup \{e\} \qquad H= D_1 \cup D_2 \cup \{e'\},\]
with $e$ an edge from $V_1$ to $V_2$ and $e'$ an edge from $V_2$ to $V_1$.\\
It easy to see that the digraphs $D$ and $H$ are not isomorphic while $\Pi(D)= \Pi(D_1) \cup \Pi(D_2)=\Pi(H)$.}

We finish this section presenting an example. As mentioned above, it is well known that there exist non isomorphic cospectral digraphs. Let us take one example of such digraphs, and compute their complementarity spectra.

Consider the two strongly connected digraphs in Figure \ref{fig:coespectrales}, which are non isomorphic \cite{Harary1971}.

\begin{figure}[h!]
\centering
\includegraphics[scale=0.15]{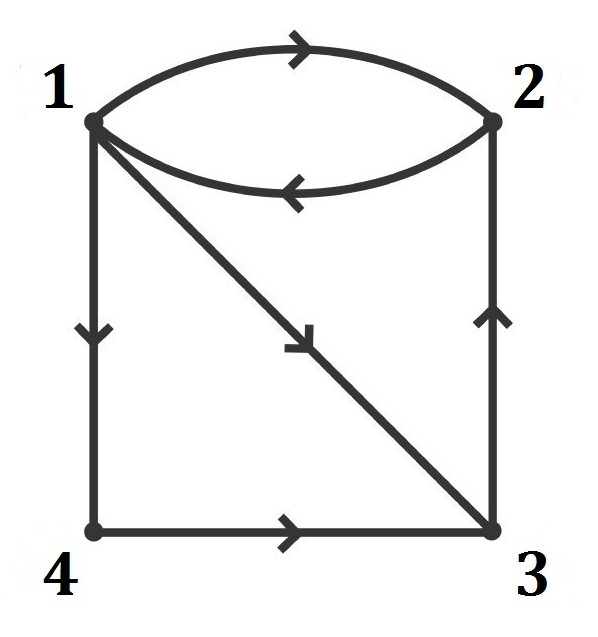} \hspace{1cm}
\includegraphics[scale=0.15]{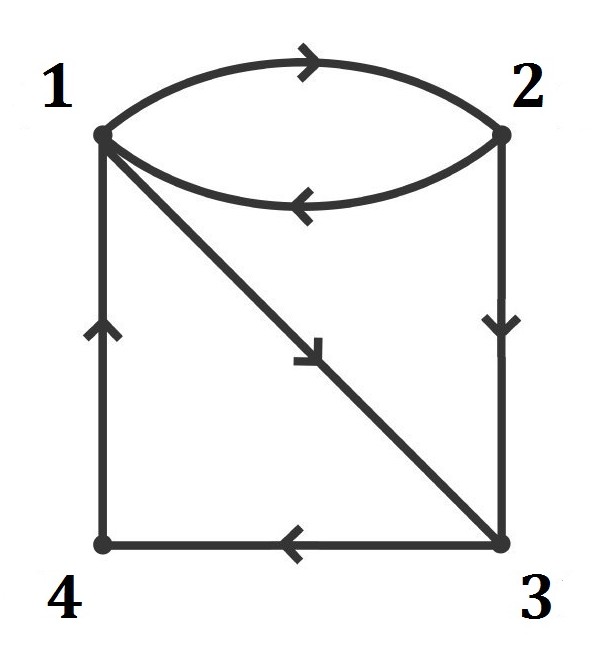}
\caption{Non isomorphic cospectral digraphs $D$ and $H$, which are not complementarity cospectral.}
\label{fig:coespectrales}
\end{figure}

{
The characteristic polynomial for both of them is $p(x) = x^4-x^2-x-1$. We observe that the digraph $D$ on the left, contains the following strongly connected induced subdigraphs: $D_0$, the subdigraph induced by the vertices $\{1,2\}$, $D_1$, induced by the vertices $\{1,2,3\}$, in addition to $D$ itself. Since $D_0$ is a cycle, its spectral radius is one. Moreover, as $D_0$ is a proper subdigraph of $D_1$ and $D_1$ a proper subdigraph  of $D$, it follows, by Lemma \ref{lem:sub}, that $1<\rho(D_0)<\rho(D)$, and hence the complementarity spectrum of $D$ is
\[\Pi(D)=\{0,1,\rho(D_0),\rho(D)\}.\]
}

On the other hand, the only strongly connected subdigraphs of the digraph on the right $H$, are $H$ itself, and  the cycles induced by vertices $\{1,2\}$ and $\{1,3,4\}$. Hence (invoking Lemma \ref{lem:sub}) its complementarity spectrum is
\[\Pi(H)=\{0,1,\rho(H)\}.\]

{Therefore, while the eigenvalues do not distinguish these two digraphs, the complementarity spectrum does. In the next section, we discuss the question of whether the complementarity spectrum distinguishes non isomorphic strongly connected digraphs.}

\section{Characterization of digraphs with at most two complementarity eigenvalues}
\label{sec:characterization}

In the previous section we have defined the complementarity spectrum of a digraph, and obtained a result that allows one to characterize it in terms of the spectral radius of  (strongly connected) induced subdigraphs.

In what follows we fully describe the digraphs with one and two complementarity eigenvalues.
{Even though the focus is on the analysis of strongly connected digraphs, our characterization includes general digraphs.}

First, we show that the existence of a cycle implies the existence of an \textit{induced} cycle. This is important because it allows one to use Theorem \ref{compl_spect_induced_subdigraphs}.

\begin{lema}
\label{digraph_induced_digraph}
Let $D$ be a digraph. If $D$ contains a cycle as a subdigraph, then $D$ contains a cycle as an induced subdigraph.
\end{lema}
\begin{proof} Let $\vec{C}$ be a cycle in $D$. We use induction on the length of $\vec{C}$.
If $length(\vec{C})=2$, since $D$ is simple, then $\vec{C}$ is generated by its vertices.
Let us suppose that the statement is true for $n=1, \hdots, k-1$, and that $length(\vec{C}) =k$. If $\vec{C}$ is not an induced subdigraph, then there exists a cycle of length $l$ in $D$, with $l<k$, and therefore, by the inductive hypothesis, there exists a cycle as an induced subdigraph of $D$.
\end{proof}

The following result connects the complementarity spectrum with the existence of cycles in a digraph.

\begin{prop}
\label{uno_comple_eigen}
Let $D$ be a digraph. Then $D$ contains a cycle if and only if $1$ is a complementarity eigenvalue.
\end{prop}
\begin{proof}
Suppose that $D$ contains a cycle. Then, due to Lemma \ref{digraph_induced_digraph}, there exists a cycle $\vec{C}$ as an induced subdigraph. Then $\rho(\vec{C})=1$ is a complementarity eigenvalue.

Conversely, if $D$ were acyclic, then the strongly connected  components of $D$ would be isolated vertices, and hence $\Pi(D)=\{0\}$, which is a contradiction.
\end{proof}

\begin{theorem}\label{tm:char} Let $D$ be a digraph and $\Pi(D)$ its complementarity spectrum. The three statements in $1$ are equivalent to each other, and the three statements in $2$ are equivalent to each other.
\begin{enumerate}
\item
\begin{enumerate}
    \item[(i)] $\Pi(D)=\{0\}$,
    \item[(ii)] $\#\Pi(D)=1$,
    \item[(iii)] $D$ is acyclic.
    \end{enumerate}
\item
    \begin{enumerate}
    \item[(i)] $\Pi(D)=\{0,1\}$,
    \item[(ii)] $\#\Pi(D)=2$,
    \item[(iii)] $D$ is not acyclic and its strongly connected components are either cycles or isolated vertices.
    \end{enumerate}
\end{enumerate}

\end{theorem}

\begin{proof}
\begin{enumerate}
\item
The implication $(i) \Rightarrow (ii)$ is trivial. If $\#\Pi(D)=1$, since zero is always a complementarity eigenvalue, we have that $\Pi(D)=\{0\}$, and by Proposition \ref{uno_comple_eigen} we conclude that $D$ is acyclic. Hence $(ii) \Rightarrow (iii)$. To see that $(iii)\Rightarrow (i)$, we notice that if $D$ is acyclic, hence the digraphs generated  by the strongly connected components of $D$ are isolated vertices, and therefore $\Pi(D)=\{0\}$.

\item We will prove the equivalence for strongly connected digraphs. The generalization is straightforward by using Proposition \ref{prop:connect_components}.
The implication  $ (i) \Rightarrow (ii)$ is obvious. To see that
$(ii) \Rightarrow (iii)$, we observe that, since $D$ is strongly connected, it contains a cycle as a subdigraph. Now, by Lemma \ref{digraph_induced_digraph}, it contains a cycle $\vec{C}$ as an induced subdigraph. If $D$ were different from $\vec{C}$, then {Lemma \ref{lem:sub} implies that} its spectral radius would be strictly greater than one.  Thus, $\{0,1,\rho(D)\}\subset \Pi(D)$, which contradicts $\#\Pi(D)=2$. Therefore, $D=\vec{C}$. The implication $(iii)\Rightarrow (i)$ follows from the fact that $\Pi(\vec{C})=\{0, \rho(\vec{C})\}=\{0,1\}$.
\end{enumerate}
\end{proof}

We emphasize that these results relate the complementarity spectrum with the cyclic structure of the digraph. In a more general way, we see them as results showing that the complementarity spectrum determines structural properties of the digraph.

So far, we have classified the strongly connected digraphs with complementarity spectrum consisting in one or two elements. In the following section we will focus on strongly connected digraphs with three complementarity eigenvalues.

\section{Non isomorphic strongly connected digraphs sharing complementarity spectrum}
\label{sec:contraejemplos}

In what follows we show examples of non isomorphic strongly connected digraphs with the same complementarity eigenvalues. This answers negatively the question of whether the complementarity spectrum characterizes digraphs, among the family of strongly connected digraphs.

{We first exhibit two families of strongly connected digraphs that have a set of  parameters. By analysing different parameters for these two families, we are able to find non isomorphic strongly connected digraphs with different number of vertices sharing their complementarity spectra. Since it is usual to require equal number of vertices for spectral determination, these digraphs are still DCS.  In the second part of this section, we do find non isomorphic strongly connected digraphs having the same number of vertices, the same number of arcs, and sharing the complementarity cospectral, showing that the complementarity spectrum is not sufficient to distinguish digraphs, even among the strongly connected ones.}

The examples given below have three complementarity eigenvalues. The simplest examples of digraphs with $\#\Pi(D) = 3$ are the following two, which also appear in \cite{Lin2012} for instance.

We will represent digraphs with figures using the following convention: a single arrow between two vertices indicates one arc joining them, while a double arrow indicates that there may be other vertices in the path joining them.

The first family of digraphs we consider is the coalescence of two cycles, which we denote by ${\infty}(r,s)=\vec{C_r}\cdot\vec{C_s}$ or simply $\infty$, following the notation which appears in \cite{Lin2012}. We will denote $V(\vec{C}_r)=\{1, \hdots, r\}$ and $V(\vec{C}_s)=\{1', \hdots, s'\}$ identifying $1$ and $1'$ in ${\infty}(r,s)$.
\begin{figure}[h!]
\centering
\includegraphics[scale=0.2]{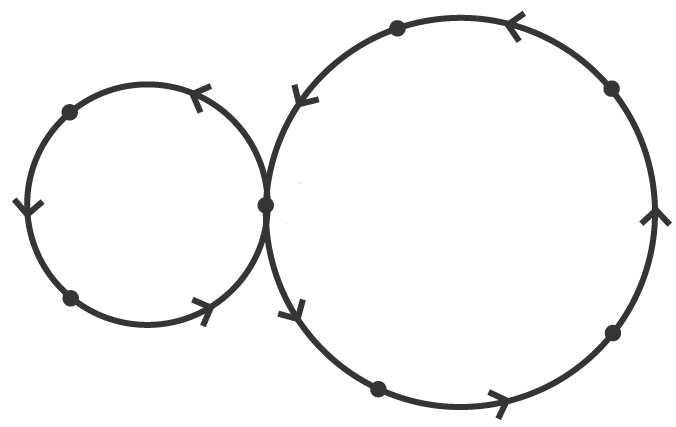}
\hspace{1.5cm}
\includegraphics[scale=0.175]{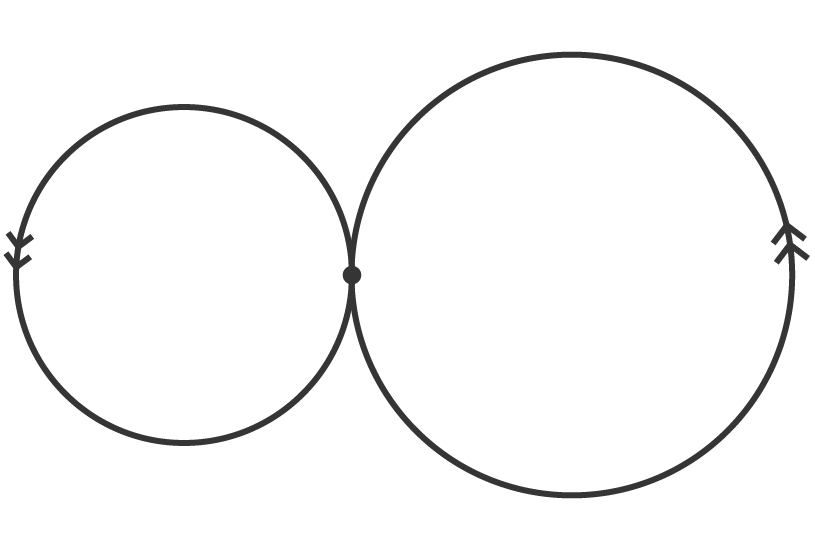}
\caption{Coalescence of two cycles: ${\infty}(3,5) = \vec{C_3}\cdot\vec{C_5}$, and schematic representation of the same digraph.}
\label{fig:ochoa}
\end{figure}

Since $\infty(r,s)$ and $\infty(s,r)$ are isomorphic we can assume $r\leq s$. Figure \ref{fig:ochoa} shows the digraph, and the above mentioned schematic representation.
It is easy to see that the only strongly connected induced subdigraphs are the cycles $\vec{C_r}$ and $\vec{C_s}$, in addition to the digraph ${\infty}$ itself and isolated vertices. Therefore, by Theorem \ref{compl_spect_induced_subdigraphs} and Lemma \ref{lem:sub}, the complementarity spectrum can be computed by means of the spectral radii of these induced subdigraphs and we have
\[\Pi({\infty})=\{0,1,\rho({\infty})\}.\]

The second family of digraphs considered is the $\theta$-digraph \cite{Lin2012} which consists of three directed paths $\vec{P}_{a+2}, \vec{P}_{b+2}, \vec{P}_{c+2}$ such that the initial vertex of $\vec{P}_{a+2}$ and $\vec{P}_{b+2}$ is the terminal vertex of $\vec{P}_{c+2}$, and the initial vertex of  $\vec{P}_{c+2}$ is the terminal vertex of $\vec{P}_{a+2}$ and $\vec{P}_{b+2}$, as shown in Figure \ref{fig:prohibido}. It will be denoted by $\theta(a, b, c)$ or simply by $\theta$.

\begin{figure}[h!]
\centering
\includegraphics[scale=0.15]{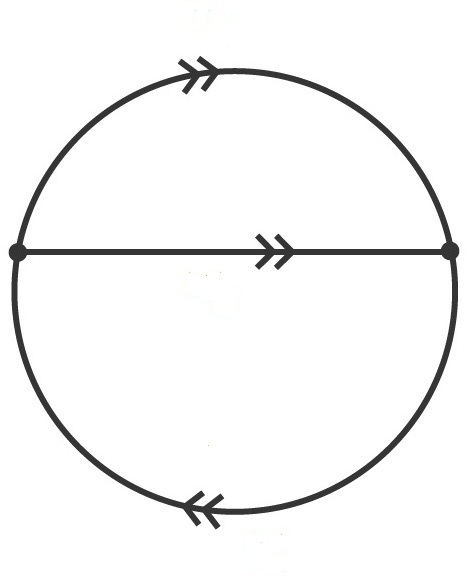}
\caption{Digraph ${\theta}(a,b,c)$.}
\label{fig:prohibido}
\end{figure}

Since ${\theta}(a,b,c)$ and ${\theta}(b,a,c)$ are isomorphic and we are not considering digraphs with multiple arcs, we can assume $a \leq b$ and $b > 0$.

We have that the number of vertices is $n=a+b+c+2$, and the only strongly connected induced subdigraphs are the cycles $\vec{C_r}$ and $\vec{C_s}$ (where $r=a+c+2$ and $s=b+c+2$), in addition to the digraph $\theta$ itself, and isolated vertices. Therefore, we have,
\[\Pi({\theta})=\{0,1,\rho({\theta})\}.\]

We may use Sachs' Theorem \ref{tm:sach} to compute the characteristic polynomial, which results in $P_{{\infty}}(x)=x^n-x^{n-r}-x^{n-s}$, where $n=r+s-1$ is the total number of vertices. Let us suppose, without loss of generality, that $r\leq s$, then,
\begin{equation}\label{eq:pinf}
 P_{{\infty}}(x)=x^n-x^{n-r}-x^{n-s}=x^{n-s}(x^s-x^{s-r}-1).
\end{equation}
By using Sachs' Theorem again, we obtain the characteristic polynomial of $P_{{\theta}}$, which is
\begin{equation}\label{eq:ptheta}
P_{{\theta}}(x) = x^n - x^{b} - x^{a}.
\end{equation}

We observe that these two digraphs just presented have $0$, $1$, in addition to a third complementarity eigenvalue, which is the spectral radius of the digraph or, equivalently, the largest root of their characteristic polynomials given by equations \eqref{eq:pinf} and \eqref{eq:ptheta}, respectively. {Since the polynomials are very similar, it is not difficult to obtain pairs of digraphs $ \infty$ and $\theta$ having the same largest root. For example, if we take $r=2, ~s=4,~b=2,~a=c=0$, we have $P_{{\theta}}(x) = x^4 - x^2 - 1$ and $P_{{\infty}}(x)=x^5-x^{3}-x=x(x^4-x^2-1)$, we see that complementarity spectrum is the same.}

{In fact, in what follows we will take advantage of this similarity to find infinitely many pairs of non isomorphic digraphs with the same complementarity eigenvalues.}

\begin{prop} \label{prop:cospec1} Let $2\leq r \leq s$, $b>0$, and $0\leq a,c$,  be integers. We have
\begin{enumerate}
  \item [$(a)$] For every $\theta(a,b,c)$ digraph there exists an $\infty(r,s)$ digraph such that $\Pi(\theta(a,b,c)) = \Pi(\infty(r,s))$.
  \item [$(b)$] For every $\infty(r,s)$ digraph there exist $r-1$ non isomorphic  $\theta(a,b,c)$ digraphs such that $\Pi(\infty(r,s))=\Pi(\theta(a,b,c))$.
\end{enumerate}
\end{prop}
\begin{proof} Item (a) follows from the polynomial equality
$x^{c+1}(x^{a+b+c+2} - x^b - x^a) =x^{a+b+2c+3} - x^{b+c+1} - x^{a+c+1}$ which may be seen as $x^{1+c}P_{\theta(a,b,c)}(x)= P_{\infty(a+c+2,b+c+2)}(x)$.\\
Item (b) follows from the fact that $x^{r+s-1} - x^{s-1} -x^{r-1} = x^i(x^{r+s-i-1} - x^{r-i-1})$, for any $i \in \{1,\ldots, r-1\}$. Now this can be seen as $P_{\infty(r,s)}(x)= x^i P_{\theta(r-i-1,s-i-1,i-1)}(x)$
  \end{proof}
\vspace{.5cm}

{By varying $c$ in Proposition \ref{prop:cospec1} (a) we see that there exist infinitely many $\theta(a,b,c)$  digraphs that have a non isomorphic complementarity cospectral mate $\infty$ digraph. Furthermore item (b) implies that, upon varying $r$, there exists infinitely many $\infty(r,s)$ digraphs having $r-1$  non isomorphic complementarity coespetral mates $\theta$-digraphs.  We notice that the digraphs in Proposition \ref{prop:cospec1} have different order, that is, have a distinct number of vertices.}

{A different infinite set of non isomorphic complementarity cospectral pairs of digraphs are obtained next.
\begin{prop}\label{prop:cospe2}
For any integer $r \geq 2$, $\infty(r,5r)$ and $\infty(2r,3r)$ have the same complementarity spectrum.
 \end{prop}
\begin{proof} It is easy to check that
\[P_{\infty(r,5r)}(x)=\frac{(x^{2r}-x^r+1)}{x^r}P_{\infty(2r,3r)}(x).\]
Since the polynomial $x^{2r}-x^r+1$ does not have real roots, we conclude that $\infty(r,5r)$ and $\infty(2r,3r)$ share the largest (real) root of their characteristic polynomial and hence their spectral radius.
\end{proof}
}

{We observe that, even though we have shown an abundant number of $\infty$ and $\theta$ digraphs with the same complementarity spectrum, they all have different orders. Since, as usual, it is required equal number of vertices for analysing spectral determination, these examples do not consist of digraphs that are not DCS.}

\subsection{Not all strongly connected digraphs are DCS}

A more challenging problem is to find two non isomorphic strongly connected digraphs with the same order (number of vertices) and size (number of arcs), and the same complementarity spectrum. {In what follows, we first give an example of such a pair of digraphs, and then we give two families, containing several pairs of non isomorphic digraphs with the same complementarity spectrum.}

\noindent {\textbf{A first example}}

Consider the two digraphs in Figure \ref{fig:ddt}, which are also taken from \cite{Harary1971} as those in Figure \ref{fig:coespectrales}\footnote{{Note that the digraph at the left is the same in both figures, but the digraph at the right is slightly different.}}. The digraph $D^T$ is obtained by reversing the arrows of digraph $D$, and therefore their adjacency matrices are transpose of each other. However, as observed in \cite{Harary1971}, these two digraphs are not isomorphic (or not self-converse). Since, from Corollary \ref{cor:Dt}, we have that the complementarity spectrum of a matrix and its transpose is the same, it follows that $\Pi(D) = \Pi(D^T)$.

{Observe that this fact is more general, and therefore a tool for generating examples of non-isomorphic digraphs sharing their complementarity spectrum, in the following sense.  Corollary \ref{cor:Dt} implies that if $D^T$ is the converse digraph of $D$, i.e., the digraph obtained by reversing all its arrows, then $\Pi(D) = \Pi(D^T)$. Hence, every pair of not self-converse digraphs is an example of non-isomorphic complementarity cospectral digraphs. }

\begin{figure}[h!]
\centering
\includegraphics[scale=0.15]{imagenes/D.jpg} \hspace{1cm}
\includegraphics[scale=0.15]{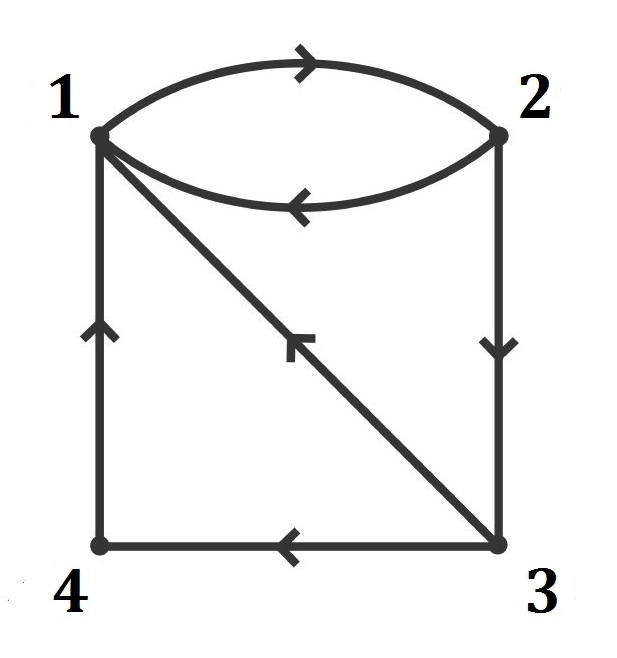}
\caption{Non isomorphic, complementarity cospectral digraphs $D$ and $D^T$.}
\label{fig:ddt}
\end{figure}

\noindent \textbf{Family 1}

Consider the digraph ${\infty}(r,s)$ described above, and let us add the arc $e=(r,2')$, connecting both cycles, as illustrated in figure.

\begin{figure}[h!]
\centering
\includegraphics[scale=0.13]{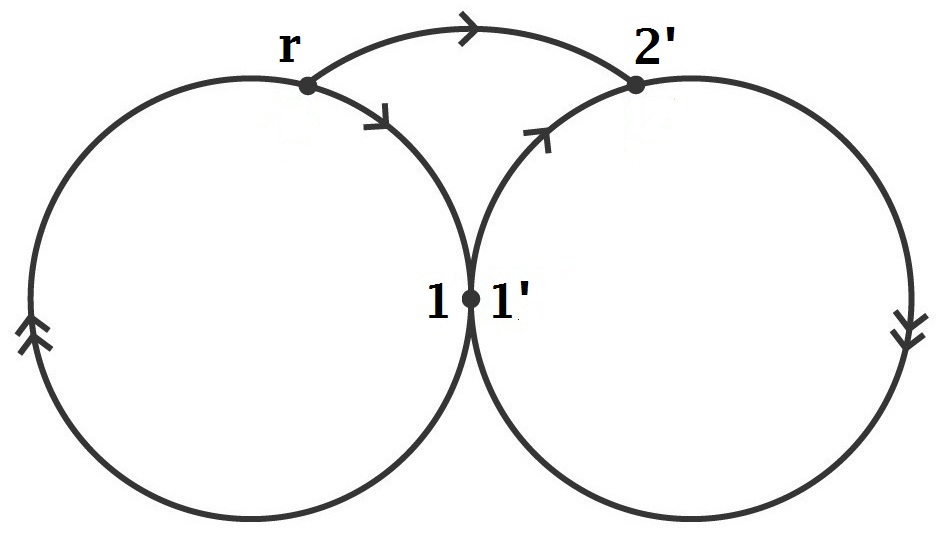}
\caption{$\widehat{\infty}(r,s)$ digraph}
\label{fig:tipo1}
\end{figure}

In the resulting digraph which will be denoted by ${\widehat{\infty}}(r,s)={\infty}(r,s)\cup\{e\} =\vec{C_r}\cdot\vec{C_s}\cup\{e\}$, the only strongly connected induced subdigraphs are the two cycles $\vec{C_r}$ and $\vec{C_s}$, in addition to the digraph ${\widehat{\infty}}(r,s)$ itself and isolated vertices. Then, we have
\[\Pi({\widehat{\infty}})=\{0,1,\rho({\widehat{\infty}})\}.\]

We can compute the characteristic polynomial using Sachs' Theorem \ref{tm:sach}, resulting in $P_{{\widehat{\infty}}}(x)=x^n-x^{n-r}-x^{n-s}-1$.

Observe that in this case, the added arc distinguishes both cycles, and therefore in general $\widehat{\infty}(r,s)$ and $\widehat{\infty}(s,r)$ are not isomorphic if $r\neq s$. In other words, adding an arc from the smaller cycle to the larger one it is not the same that adding an arc from the larger to the smaller. On the other hand, since the characteristic polynomial depends only on the lengths of the cycles, both digraphs have the same characteristic polynomial, and therefore the same spectral radius. We conclude that $\widehat{\infty}(r,s)$ and $\widehat{\infty}(s,r)$ have the same complementarity spectrum, while they are non isomorphic in general. We have proven the following result.
{
\begin{theorem}\label{tm:notdcs1} For any pair $(r,s)$, with $2\leq r < s$, the non-isomorphic digraphs $\widehat{\infty}(r,s)$ and $\widehat{\infty}(s,r)$ have the same order $n=r+s-1$, the same size $m=r+s+1$, and the same complementarity spectrum $\Pi({\widehat{\infty}})=\{0,1,\rho({\widehat{\infty}})\}$, where $\rho({\widehat{\infty}})$ is the largest root  of $x^n-x^{n-r}-x^{n-s}-1$.
\end{theorem}
}

\noindent \textbf{Family 2}

Consider now the cycle $\vec{C_{n}}$, with $n\geq 4$. Now consider $j\geq 2$ and let $D(j)$ be the digraph obtained by adding the arcs $(2,1)$ and $(j+1,j)$ to it. A representation of this digraph can be seen in Figure \ref{fig:tipo4}.

\begin{figure}[h!]
\centering
\includegraphics[scale=0.15]{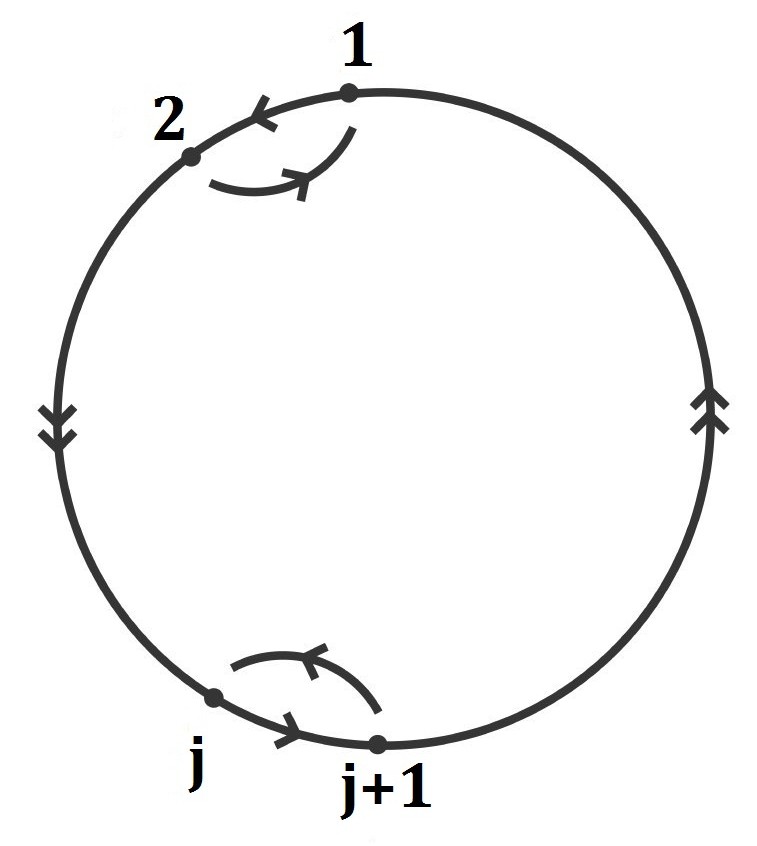}
\caption{Digraph $D(j)$}
\label{fig:tipo4}
\end{figure}

This digraph contains, as induced subdigraphs, the two cycles $\vec{C_2}$, and the digraph $D(j)$ itself, in addition to isolated vertices. Therefore the complementarity spectrum is
\[\Pi(D)=\{0,1,\rho(D)\}.\]

In order to compute the spectral radius, observe that this digraph contains as linear subdigraphs the two cycles $\vec{C_2}$, its direct sum, and the original $\vec{C_n}$. Therefore, using Sachs' Theorem \ref{tm:sach} we can compute the characteristic polynomial, which ends up being $P_{D(j)}(x) = x^n - 2x^{n-2} + x^{n-4} -1$. We notice that this polynomial, and in particular the spectral radius, does not depend on $j$, and then all these digraphs have the same complementarity spectrum. It is easy to see that varying $j$ we can obtain {sets} of non isomorphic digraphs, since for instance there are two induced directed paths $P_{j-1}$ and $P_{n-j+1}$, which we can use to distinguish the digraphs within the family. {We summarize our findings stating the following result.
\begin{theorem}\label{tm:notdcs2}
Let $n\geq 4$ be an integer. For all $ j \in  \{2,\ldots, n-1\}$, the digraghs $D(j)$ have the same order $n$, the same size $m=n+2$ and the same complementarity spectrum $\Pi(D)=\{0,1,\rho(D)\}$, where $\rho(D)$ is the largest root of $x^n - 2x^{n-2} + x^{n-4} -1$.
\end{theorem}}

\section{Conclusions}\label{sec:final}

In this paper we have defined the complementarity spectrum for digraphs, and characterized it in terms of the spectral radii of its strongly connected induced subdigraphs.
{We characterized the complementarity spectrum of a digraph in terms of its structure, and in particular we give a complete description of digraphs with one or two complementarity eigenvalues. For instance, zero is a complementarity eigenvalue of every digraph, and the complementarity spectrum contains $1$ as an element if and only if it contains a cycle.}
We then addressed the question of whether the complementarity spectrum characterizes digraphs, i.e., if non isomorphic digraphs have different complementarity spectrum. In order to do this, we studied some examples of strongly connected digraphs with three complementarity eigenvalues. We showed some relations between the characteristic polynomials involved, and finally we presented families of non isomorphic digraphs of arbitrary size, sharing the complementarity spectrum, answering negatively the mentioned question. We conclude that the complementarity spectrum is not sufficient to distinguish digraphs.

It is worth pointing out that the cardinality of the complementarity spectrum of digraphs is a key parameter for understanding the structure of the digraph. We observe that given an undirected graph $G$ of order $n$ its  spectrum is composed by $n$ eigenvalues (counting multiplicities), whereas the complementarity spectrum is at least $n$, being equal $n$ only for the path, the cycle, the star and the complete graph. For a digraph $D$, the complementarity spectrum may be composed by a single element, independent of its order. We finalize this paper by suggesting a few research problems.

\begin{problem} {For any integer $j \geq 1$, find strongly connected digraphs of arbitrary order $n\geq j$ whose complementarity spectrum has cardinality $j$.}
\end{problem}
\begin{problem} Characterize the digraphs whose complementarity spectrum has exactly 3 elements.
\end{problem}
\begin{problem} Find a class of digraph that is DCS.
\end{problem}

\section*{Acknowledgments} V. Trevisan  acknowledges the financial support provided by the hosting University of Naples Federico II, and by CAPES-Print 88887.467572/2019-00, Brazil that provide the support for the visiting position. V. Trevisan also acknowledges partial support of CNPq grants 409746/2016-9 and 303334/2016-9, and FAPERGS
PqG 17/2551-0001. M. Fiori and D. Bravo acknowledge the financial support provided by ANII, Uruguay. F. Cubría thanks the doctoral scolarship from CAP-UdelaR.

\bibliographystyle{elsarticle-harv}
\bibliography{biblio}





\end{document}